\documentclass[twoside]{aiml}

\usepackage{aimlmacro}

\usepackage{mathtools}
\usepackage{mathrsfs}
\usepackage{enumitem}
\usepackage{cleveref}

\crefname{theorem}{theorem}{theorems}
\Crefname{theorem}{Theorem}{Theorem}

\crefname{lemma}{lemma}{lemmata}
\Crefname{lemma}{Lemma}{Lemmata}

\crefname{remark}{remark}{remarks}
\Crefname{remark}{Remark}{Remarks}

\crefname{corollary}{corollary}{corollaries}
\Crefname{corollary}{Corollary}{Corollaries}

\usepackage{soul}
\usepackage{bussproofs}
\usepackage{pstricks}
\usepackage{aliascnt}

\usepackage[a]{esvect}

\usepackage{algorithm}
\usepackage{algorithmic}
\usepackage{float}

\def\lcal{\mathcal{L}}

\def\lcalz{\lcal}

\def\lcalzpar{{\lcalz(\parr)}}
\def\lcalzpar{\lcalz(\parr)}
\def\kcal{\mathcal{K}}

\def\IPC{{\sf IPC}}

\def\HA{\hbox{\sf HA}{} }

\def\K4{\hbox{\sf K4}{} }

\def\IPC{{\sf IPC}}
\def\HA{\hbox{\sf HA}{} }

\def\NNIL{{\sf NNIL}}

\def\K4{\hbox{\sf K4}{} }

\def\ar{\mid\!\sim}

\newcommand{\adms}[2]{\mathrel{\ar}^{{\,{#1}}}_{\,{{#2}}}}
\def\admspar{\adms{}{\Gamma}}

\def\pr{\mid\!\approx}

\def\pce{\preccurlyeq}
\def\sce{\succcurlyeq}

\newcommand{\V}{\mathrel{V}}

\newcommand{\lr}{\leftrightarrow}

\def\scrk{\mathscr{K}}

\def\scrm{\mathscr{M}}

\def\ar{            \mathrel{\setlength{\unitlength}{1ex}
                    \begin{picture}(1.65,1.65)           
                    \put(0,0){\line(0,1){1.65}}                  
                    \put(0,.3){\textup{\footnotesize{$\sim$}}}
                    \end{picture}}\!
        }

\newcommand{\vdashm}{\Vdash^{^{\!\!-}}}

\newcommand{\vdashc}{\vdash^{{{\sf c}}}}
\newcommand{\vdashi}{\vdash^{{{\sf i}}}}

\newcommand{\nin}{\not\in}

\def\sfL{{\sf L}}

\def\atom{{\sf atom}}
\newcommand{\varr}{{\sf var}}
\newcommand{\parr}{{\sf par}}

\newcommand{\complexity}[1]{c_\to(#1)}
\def\max{{\sf max}}

\newcommand{\gampro}[2]{{#1}^{#2}}
\newcommand{\uap}[2]{\lceil{#1}\rceil}
\newcommand{\dap}[2]{\lfloor{#1}\rfloor}
\newcommand{\dclosur}[2]{\langle{#1}\rangle_{#2}}
\def\mcal{\mathcal{M}}

\def\nvdashi{\nvdash^{\sf i}}
\newcommand{\bisim}[1]{\rightleftharpoons_{#1}}
\newcommand{\gen}[2]{#1_{#2}}

\newcommand{\proj}[1]{\widehat{#1}}
\newcommand{\setpar}[1]{{{#1}^\parr}}
\def\Gpar{\setpar\Gamma}
\newcommand{\Val}[2]{#1(#2)}
\newcommand{\parpro}[1]{{#1}^*}

\def\lastname{Mojtahedi and Papafilippou}

\begin{document}

\begin{frontmatter}
  \title{Projectivity meets Uniform Post-Interpolant: \\
Classical and Intuitionistic Logic}
  \author{Mojtaba Mojtahedi}\footnote{Email: mojtahedy@gmail.com. 
  This work is partially funded by FWO grant G0F8421N and BOF grant BOF.STG.2022.0042.01.}
  \author{Konstantinos Papafilippou}\footnote{Email: Konstantinos.Papafilippou@uGent.be.
  Funded by the FWO-FWF Lead Agency grant G030620N (FWO)/I4513N (FWF) and
by the SNSF--FWO Lead Agency Grant 200021L\underline{ }196176/G0E2121N.}
  \address{Ghent University}

\begin{abstract}
We examine the interplay between projectivity (in the sense that was introduced by S.~Ghilardi) 
and uniform post-interpolant for the classical and intuitionistic propositional logic.
More precisely, we explore whether a projective substitution of a formula is equivalent 
to its uniform post-interpolant, assuming the substitution leaves the variables of the interpolant unchanged.
We show that in classical logic, this holds for all formulas. 
Although such a nice property is missing in  intuitionistic logic,
we provide Kripke semantical characterisation for propositions with this property.

As a main application of this, we show that the unification type of some extensions of intuitionistic logic are finitary.
In the end, we study admissibility for intuitionistic logic, relative to some sets of formulae.

The first author of this paper recently considered a particular case of this relativised admissibility and 
found it useful in characterising the provability logic of Heyting Arithmetic.
\end{abstract}

  \begin{keyword}
  Unification, Projectivity, Admissible Rules, Intuitionistic Logic, Uniform Interpolation.
  \end{keyword}
 \end{frontmatter}




\section{Introduction}
The notion of projectivity for propositional logics, first introduced by S.~Ghilardi \cite{Ghilardi97,Ghil99,Ghil2000modal}. 
Roughly speaking, a formula $A$ is projective if 
there is a substitution $\theta$ which unifies $A$, 
i.e.~$\vdash \theta(A)$, and $A\vdash \theta(B)\lr B$ for every formula $B$.
The importance of projective formulas is mainly due to the following observation: any projective formula, 
has a single unifier which is more general than all other unifiers. 
Then, by approximating formulas with projective ones, 
we may fully describe \textit{all unifiers} 
of a given formula \cite{Ghil99,Ghil2000modal,balbiani2021unification,alizadeh2023unification}.  
Having this strong tool in hand, 
we may then characterise admissible 
rules\footnote{An inference rule $A/B$ is admissible if every unifier of $A$ is also 
a unifier of $B$.}
\cite{Iemhoff-admissibility,jevrabek2005admissible,jevrabek2010admissible,cintula2010admissible,GOUDSMIT2014652,VANDERGIESSEN2023103233} for a logic. 
The situation for classical logic is simple: Every satisfiable formula is also projective. Hence no approximation is required and also no non-derivable admissible rule exists ($A/B$ is admissible iff $A\to B$ is derivable).

Additionally, we may consider propositional language with two sorts of
atomics: variables $\varr$ and parameters (constants) $\parr$. Variables are 
intended to be substituted, while parameters are considered as constants and 
are never substituted. Again  one may wonder about the previous  questions of 
unification \cite{balbiani2021unification} and admissibility 
\cite{JERABEK2015881,jevrabek2020rules} for this two-sorted language.
But this time, even for classical logic, not all satisfiable formulas are 
projective. For example the  formula $A:=p\wedge x$ for atomic variable $x$ 
and parameter $p$  is satisfiable, while it is not unifiable. This will rule 
out $A$ being projective
only because of it not being unifiable. 
To mend this, we define the notion of $E$-projectivity by 
replacing condition 
$\vdash \theta(A)$ by $\vdash \theta(A)\lr E$ in above definition. 
This time in our example, 
the formula $A$ is $p$-projective. 
In this paper, we first observe that  
$A$ can only be $E$-projective for a unique $E$ 
(\Cref{par-projection-uniqueness}), which is annotated as $\parpro A$. 
Then we show 
(\Cref{classical-proj-all-2}) that all $A$'s are 
$\parpro A$-projective in classical logic. 
For intuitionistic logic, we find a Kripke 
semantical characterisation for those  $A$'s 
that are $\parpro A$-projective (\Cref{par-proj-par-extendable}). 

On the other hand, we have the well-studied notion of the 
\textit{uniform post-interpolant}. Given a formula $A$, the uniform 
post-interpolant of $A$ with respect to $\parr$, is the strongest formula 
$B$ in the variable-free language 
such that $A\to B$ holds.  
Of course, it might be the case 
that such a \textit{strongest} formula does not exist
at all. Fortunately, several logics
indeed have uniform post-interpolants, namely 
intuitionistic logic \cite{Visser_2017} and
all locally tabular (finite) logics like 
classical logic. Interestingly, it turns out 
(\Cref{par-proj-interpol}) that $\parpro A$ 
is equal to the uniform post-interpolant of $A$.
As an application, we show that 
extensions of intuitionistic logic with variable-free
formulas have a finitary unification type. More precisely, 
in the mentioned logics, we show that every unifiable formula has a finite complete set 
of unifiers, i.e. a finite set of unifiers such that 
every unifier is less general than at least one 
of them. In the end, we observe that 
admissibility relative to the set of all 
variable-free formulas is trivial: 
$A\adms{\parr}{} B$\footnote{$A\adms{\parr}{} B$ 
is defined as follows: for every variable-free $E$
and substitution $\theta$
such that $\vdash E\to \theta(A)$ we have 
$\vdash E\to\theta(B)$.}
iff 
$\vdash A\to B$ for the intuitionistic logic.

Motivations for the study of relativised projectivity 
and admissibility come from their applications in 
the study of the provability logic of $\HA$\footnote{Heyting Arithmetic}
\cite{PLHA}. More precisely, 
the first author, studied \cite{PLHA0} projectivity
and admissibility
relative to the set $\NNIL$\footnote{$\NNIL$ 
is the set of formulas with No Nested Implications 
on the Left \cite{Visser-Benthem-NNIL}.} and used 
it in a  crucial way \cite{PLHA} to axiomatize 
the provability logic of $\HA$.
Interestingly, 
\cite{balbiani2023projective,baader2010unification} have also found 
relative unification beneficiary for their study.  Nevertheless,  they
consider the non-parametric language.  
The main aim of this paper is to look at this relativised notion 
of projectivity and admissibility, in some more general viewpoint. 
In this direction, \Cref{par-proj-approx-general} generalises the previous result in \cite{PLHA0} 
from $\NNIL$ to all finite sets of formulas $\Gamma$.
However, it appears that there is no straightforward generalisation for 
the characterisation of admissibility relative to $\NNIL$. 

\section{Definitions and basic facts}\label[section]{sec-definition}

\subsubsection*{Language}
In the sequel, we assume that $\parr$ and $\varr$ are two sets of 
parameters and variables, respectively. Unless otherwise stated, we assume that 
both sets $\varr$ and $\parr$ are infinite.
Then the propositional language $\lcalz$, is the set of all 
Boolean combinations of atomic 
propositions $\atom:=\varr\cup\parr$. 
Boolean connectives are $\bot$,  $\wedge$, $\vee$ and $\to$. 
Note that all connectives are binary, except for falsity 
$\bot$ which is nullary, i.e.~an operator without argument. 
Then $\top$, $\neg A$ and $A\lr B$ are shorthand for $\bot\to\bot$, $A\to \bot$
and $(A\to B)\wedge(B\to A)$, respectively.
Sets of formulas are indicated by 
$\Gamma$, $\Delta$, $\Pi$, $X$ and $Y$. 
We use $x,y,\ldots$ and $p,q,\ldots$ and $a,b , \ldots$ and 
$A, B, C, D, E,F,\ldots$ as meta-variables for variables, parameters,
and formulas, respectively. Optionally, we may also use subscripts for them.
Let $\lcalz(X)$ indicate all Boolean combinations of propositions in $X$. 
Hence, for example, $\lcalz(\varr\cup\parr)$ is equal to $\lcalz$. Moreover, $\Gpar$ indicates the set $\Gamma\cap\lcalzpar$.
We use $\Gamma\vdashi A$ to indicate the intuitionistic derivability of $A$ from the set $\Gamma$. We also use $\Gamma\vdashc A$ for the classical derivability. 
Whenever we state a definition for both intuitionistic and classical cases, we simply use $\vdash$ without superscript.

A \textit{$\parr$-extension} of a logic is an extension of that logic by 
adding some $E\in\lcalzpar$ to its set of axioms. Note that this additional axiom is 
not considered as a schema. 


\subsubsection*{Kripke models}
A Kripke model is a tuple $\kcal=(W,\pce,\V)$, whose frame $(W,\pce)$,  
is a reflexive transitive ordering. 
We write $w\prec u$ to $w\pce u$ and $w\neq u$. 
Moreover, we use $\sce$ and $\succ$ as inverses of the relations $\pce$ and $\prec$, respectively.
A root of $\kcal$, is the minimum element in $(W,\pce)$ if it exists, and we say that $\kcal$ is rooted if it has a root.
We say that $\kcal$ is finite if $W$ is so.
In this paper, it is assumed that all Kripke models are 
\textit{finite} and \textit{rooted}.
We use the notation $\kcal,w\Vdash A$ for the validity / forcing of $A$ at the node $w$ in the model $\kcal$.
$\kcal\Vdash A$ is defined as $\forall\,w\in W\ \big(\kcal,w\Vdash A\big)$.
Given $w\in W$, we define $\gen \kcal w$ as the restriction of $\kcal$ to all nodes that are accessible from $w$.
Furthermore, $\Val{\kcal}{w}$ indicates the valuation of $\kcal$ at $w$, i.e.~the set of all atomics $a$ such that 
$w\V a$.

\subsubsection*{Substitutions}
\textit{Substitutions} are functions $\theta:\lcalz\longrightarrow\lcalz$ with the following properties: (they are indicated by lowercase Greek letters $\theta,\tau,\lambda ,\gamma,\ldots$)
\begin{itemize}
    \item $\theta$ commutes  with connectives. This means that for binary connective $\circ$ we have $\theta(A\circ B)=\theta(A)\circ\theta(B)$ and for $\bot$ we have $\theta(\bot)=\bot$.
    \item $\theta(p)=p$ for every $p\in\parr$.
\end{itemize}
Given two substitutions $\theta$ and $\gamma$, we define $\theta\gamma$ as the composition of $\theta$ and $\gamma$: $\theta\gamma(a):=\theta(\gamma(a))$.
Given a 
substitution $\theta$ and the Kripke model $\kcal=(W,\pce,\V)$, define
$\theta^*(\kcal)$ as a Kripke model with the same frame of $\kcal$, and with  valuation $\V'$ as follows:
$w\V' a$ iff $\kcal,w\Vdash\theta(a)$.  Then it can be easily observed that 
$\kcal,w\Vdash \theta(A)$ iff $\theta^*(\kcal),w\Vdash A$, for every $w\in W$ and $A\in\lcalz$.
If no confusion is likely,\footnote{Confusions are due to this fact that 
$(\theta\gamma)^*=\gamma^*\theta^*$.} we may simply use $\theta(\kcal)$ instead of $\theta^*(\kcal)$.

Given two substitutions $\theta$ and $\gamma$, we say that $\theta$ is more general than 
$\gamma$ ($\gamma$ less general than $\theta$), annotated as $\gamma\leq \theta$, 
if there is a substitution $\lambda$ such that $\vdash \gamma(x)\lr \lambda\theta(x)$.

\subsubsection*{Ralative projectivity}
An \textit{$A$-identity} is some $\theta$ such that 
\begin{equation}\label{eq1}
    \forall\, a\in\atom \ \  A\vdash \theta(a)\lr a.
\end{equation}

Let $E$ be a proposition in the parametric language $\lcalzpar$.
An $E$-fier of $A$ is some $\theta$  such that $\vdash \theta(A)\lr E$.  
An $E$-fier for some $E\in\Gpar:=\Gamma\cap\lcalzpar$ is also called a $\Gamma$-fier. Note that here 
$\Gamma$ is not necessarily a subset of $\lcalzpar$.
$\top$-fier and $\lcalz$-fier are also called \textit{unifier} and \textit{parametrifier}, respectively.
We say that $A$ is \textit{$E$-projective} ($\Gamma$-projective), 
 if there is some $\theta$ which is $A$-identity 
 and $E$-fier ($\Gamma$-fier) of $A$.
 In this case we say that $\theta$ is  a projective 
 $E$-fier ($\Gamma$-fier) of $A$ and $E$ is \textit{the} projection of $A$. 
As we will see in \Cref{par-projection-uniqueness}, the projection is unique, modulo provable equivalence. 
 So we use the notation $\parpro A$ for its unique projection, if it exists. 
Then $\lcalz$-projectivity is also called $\parr$-projectivity.
The well-known notion of \textit{projectivity}, as introduced by S.~Ghilardi \cite{Ghilardi97,Ghil99}, 
coincided with $\top$-projectivity in this general setting.

Given $A\in\lcalz$ and a set $\Theta$ of substitutions, 
we say that $\Theta$ is a complete set of unifiers of $A$, if the following holds:
(1) every $\theta\in\Theta$ is a unifier of $A$, (2) every unifier of $A$ is less general 
than some $\theta\in\Theta$.
We say that the unification type of a logic is:
\begin{itemize}
\item unitary, if every unifiable formula of the 
logic, has a singleton complete set of unifiers,
\item finitary, if every unifiable formula of the 
logic, has a finite complete set of unifiers.
\end{itemize} 
Note that in our definition, if the unification type of a logic is unitary, it is also finitary.

All the above definitions of projectivity rely on a \textit{background logic} in the context. 
More precisely, we have two different notions of $\Gamma$-projectivity: one for classical logic and one for intuitionistic logic. 
To simplify the notation, we always hide the dependency on the logic. 
This will not cause much confusion, since it is quite clear from the context which logic we are talking about. 
To be more precise, in this section, we provide all statements for both classical and intuitionistic logic.
Later in \cref{sec-class}, classical logic is our background logic, whereas
for the rest of the paper, the background logic is intuitionistic logic. 

 \begin{lemma}[\textbf{Uniqueness of projections}]\label[lemma]{par-projection-uniqueness}
Projection is unique, modulo provable equivalence.
 \end{lemma}
\begin{proof}
    Let $\vdash \theta_i(A)\lr B_i$ and $A\vdash \theta_i(x)\lr x$ and $B_i\in\lcalzpar$ for $i=1,2$ and every $x\in\varr$. 
    Then $A\vdash \theta_1(A)\lr A$ and thus 
    $\vdash A\to B_1$. Therefore, we have 
    $\vdash \theta_2(A\to B_1)$ and since 
    $\theta_2$ is identity 
    over the language $\lcalz(\parr)$, 
    we may deduce $\vdash \theta_2(A)\to B_1$, 
    and thus $\vdash B_2\to B_1$. By a similar 
    argument we may prove  $\vdash B_1\to B_2$, 
    and thus $\vdash B_1\lr B_2$.
\end{proof}

%
%
%

\subsubsection*{Uniform post-interpolant and upward approximations}
Given a formula $A$, we say that $B$ is its uniform post-interpolant with respect to $\parr$, 
if we have (1) $B\in\lcalzpar$, (2) $\vdash A\to B$ and (3) for every $C\in\lcalzpar$ s.t.~$\vdash A\to C$ we have $\vdash B\to C$. 
It can be easily proved that such a $B$ is unique (modulo provable equivalence), if it exists. 
Hence, we use the notation $\uap{A} \parr$ for the uniform post-interpolant of $A$ with respect to $\parr$.
It is a well-known fact that the uniform post-interpolant for classical and intuitionistic logic exists. 
In the following theorem, we will show that the unique projection is equal to the 
uniform post-interpolant with respect to the set $\parr$ of atomics. 
\begin{theorem}\label[theorem]{par-proj-interpol}
For every $\parr$-projective $A$, we have $\parpro A =\uap A \parr$. 
\end{theorem}
\begin{proof}
   Let  $A\vdash \theta(x)\lr x$ for every $x\in\varr$ and $\vdash\theta(A)\lr \parpro A$ and 
   $\parpro A \in\lcalzpar$.
   We will check that all 3 required conditions for it to be a uniform post-interpolant hold. 
(1) Trivially, we have $\parpro A\in\lcalzpar$. 
(2)  By the argument provided in the proof of \Cref{par-projection-uniqueness}, 
we have $\vdash A\to \parpro A$. (3)  Let $C\in\lcalzpar$ s.t.~$\vdash A\to C$. 
Hence $\vdash \theta(A)\to \theta(C)$ and since $C\in\lcalz(\parr)$ we get $\theta(C)=C$. 
Thus, $\vdash \parpro A\to C$.
\end{proof}

Given $A\in\lcalz$, one might wonder if $\parr$-projectivity  of $A$ could be defined using the standard notion of
projectivity. We will prove that $\parr$-projectivity of $A$ is equivalent to projectivity of $A\lr \uap A\parr$, for both classical and intuitionistic logic.
One direction is easy and stated in the following lemma. However, for the other direction,
we provide separate proofs for intuitionistic and classical logic. For intuitionistic logic,
we take advantage of the semantic characterisation of $\parr$-projectivity (see \Cref{lem-par-proj-proj}). 
For classical logic, \Cref{c1} and \Cref{c2} imply that $A\lr \uap A\parr$ is projective for every $A$!
\begin{lemma}\label[lemma]{lem-proj-par-proj}
    Every projective unifier of $A\lr \uap A\parr$ is also a projective 
    parametrifier of $A$. Therefore,
    projectivity of $A\lr \uap A\parr$ implies $\parr$-projectivity of $A$.
\end{lemma}
\begin{proof}
    Let  $\theta$ be a projective unifier of $A\lr \uap A\parr$.
    Hence $\vdash\theta(A\lr \uap A\parr)$ and $\vdash \theta(A)\lr  \uap A\parr$. 
    This shows that $\theta$ is a parametrifier of $A$.
    Also, for every variable $x$ we have $A\lr \uap A\parr\vdash \theta(x)\lr x$.  
    Since $\vdash  A\to \uap A\parr$, we have $\uap A\parr\to A\vdash \theta(x)\lr x$.  
    Hence, a fortiori, we have $ A\vdash \theta(x)\lr x$. 
    Thus, $\theta$ is a projective parametrifier of $A$.  
\end{proof}

\begin{theorem}\label[theorem]{mgparametrifier}
Every projective unifier of $A\lr \uap A\parr$ is a most general 
$\uap A\parr$-fier of $A$.
\end{theorem}
\begin{proof}
Let $\theta$ be a projective unifier of $A\lr \uap A\parr$. 
It is obvious that $\theta$ is also a $\uap A\parr$-fier of $A$.
To show that it is a most general $\uap A\parr$-fier of $A$,
let $\gamma$ be some $\uap A\parr$-fier of $A$. 
Hence $\gamma$ is also a unifier of $A\lr \uap A\parr$. Since 
projective unifiers are most general unifiers, we may infer that 
$\gamma$ is less general than $\theta$.
\end{proof}
\section{\texorpdfstring{$\lcalz$}{L}-Projectivity for Classical Logic}\label{sec-class}
In this section we will prove that in Classical Logic, every $A\in\lcalz$ is $\parr$-projective. First, let us see some definitions. 

We say that an atomic $x$ is positive (negative) in $A$ if 
for all interpretations $I$ and $J$ such that $I(x)\leq J(x)$ and 
$I(a)=J(a)$ elsewhere,
we have $I(A)\leq J(A)$ ($I(A)\geq J(A)$). 
We use the order $\leq$ on truth-falsity 
values so that falsity is less than truth. 
\begin{lemma}\label[lemma]{posit-eqiv}
A variable $x$ is positive in $A$ iff $\vdashc A\to A[x:\top]$.    
\end{lemma}
\begin{proof}
    Left-to-right: Let $x$ be positive in $A$ and $I\models A$
    seeking to show that $I\models A[x:\top]$.
    Let $J$ be the interpretation that is always equal to $I$ except for $J(x):=\top$.  Then obviously $I(x)\leq J(x)$ and thus 
    $I(A)\leq J(A)$. Since $I\models A$ we get $J\models A$ and thus
    $I\models A[x:\top]$.\\
    Right-to-left: Let $\vdashc A\to A[x:\top]$, 
    $I(x)< J(x)$ (note that this implies $J\models x$) and $I(a)=J(a)$ elsewhere. 
    Assume $I\models A$ to show $J\models A$. 
    Since $I\models A\to A[x:\top]$ we get $I\models A[x:\top]$, which is just $J\models A$.
\end{proof}

Given $A\in\lcalz$ and some $x\in \varr$, define the substitution  $\theta^A_x$ as follows:
$$\theta^A_x(y):=\begin{cases}
y \quad &: y\neq x\\
A\wedge y &: y=x
\end{cases}$$

\begin{lemma}\label[lemma]{lem-posit} 
 Given $A$, 
   the variable $x$ is positive in $\theta^A_x(A)$.
   Furthermore, if $y$ is positive in $A$, then it is also positive 
   in $\theta^A_x(A)$.
\end{lemma}
\begin{proof}
 Let $I\not\models x$ and $J\models x$ and $I(a)=J(a)$ 
 for every atomic $a\neq x$. 
 Moreover, assume that $I\models \theta^A_x(A)$, 
 seeking to show $J\models \theta^A_x(A)$. If $I(A)=J(A)$, then 
 obviously $I(\theta^A_x(A))=J(\theta^A_x(A))$ and we are done.
 Hence, we have the following cases:
 \begin{itemize}
     \item $I\models A$ and $J\not\models A$. Then we have 
     $I(\theta^A_x(x))=J(\theta^A_x(x))$ and thus 
     $I(\theta^A_x(A))=J(\theta^A_x(A))$. 
     \item $I\not\models A$ and $J\models A$.
     Then we have $J(\theta^A_x(x))=J(x)$ and thus 
     $J(\theta^A_x(A))=J(A)$. This means that $J\models \theta^A_x(A)$.
 \end{itemize}
 At the end, we must show that if $y$ is positive in $A$, then it is 
 also positive in $\theta^A_x(A)$. 
  Let $I\not\models y$ and $J\models y$ and $I(a)=J(a)$ 
 for every atomic $a\neq y$. 
 Moreover, assume that $I\models \theta^A_x(A)$, 
 seeking to show $J\models \theta^A_x(A)$.
 If $I(A)=J(A)$, then 
 obviously $I(\theta^A_x(A))=J(\theta^A_x(A))$ and we are done.
 So we have the following cases:
 \begin{itemize}
     \item $I\models A$ and $J\not\models A$.
     This case contradicts the positiveness of $y$ in $A$.
     \item $I\not\models A$ and $J\models A$.
     Then we have $J(\theta^A_x(x))=J(x)$ and thus 
     $J(\theta^A_x(A))=J(A)$. This means that $J\models \theta^A_x(A)$.
 \end{itemize}
\end{proof}

\begin{lemma}\label[lemma]{c1}
The following items are equivalent:
\begin{enumerate}
    \item $A$ is unifiable.
    \item $A$ is projective.    
    \item $\uap A \parr=\top$.
\end{enumerate}
\end{lemma}
\begin{proof}
(i)$\Rightarrow$(ii): Let $\vdash\tau(A)$ and define the substitution $\epsilon_\tau$ as follows:
$$\epsilon_\tau(x):= (A\wedge x)\vee (\neg A\wedge \tau(x)) \quad \text{for every variable }x.$$
It is obvious that $\epsilon_\tau$ is $A$-identity. Moreover one may easily prove that both 
$A\vdashc \epsilon_\tau(A)$ and $\neg(A)\vdashc \epsilon_\tau(A)$ and thus $\vdashc \epsilon_\tau(A)$. 
\\[2mm]
(ii)$\Rightarrow$(iii): This holds by \Cref{par-proj-interpol}.
\\[2mm]
(iii)$\Rightarrow$(i):
We use induction on $n_A$, the number of non-positive variables occurring in $A$. 
First, observe that if $n_A=0$,
then for the substitution $\tau$ which replaces all variables with $\top$, by \Cref{posit-eqiv} we have
$\vdash A\to\tau(A)$ and $\tau(A)\in\lcalzpar$. Then, since $\uap A\parr=\top$, we get 
$\vdashc \tau(A)$, and thus $A$ is unifiable.
As induction hypothesis, assume that every formula $B$ with $n_B=n_A-1$ and $\uap B \parr=\top$ is unifiable.
Take some positive variable $x$ in $A$ and let $B:=\theta^A_x(A)$. Since 
$A\vdash \theta^A_x(y)\lr y$ for every $y$, we get $\vdash A\to B$ and thus $\uap B \parr=\top$. 
Moreover, \Cref{lem-posit} implies that $n_B=n_A-1$ and hence by the induction hypothesis, there is a substitution
$\gamma$ which unifies $B$. Therefore, $\gamma\circ\theta^A_x$ unifies $A$, as desired.
\end{proof}

\begin{lemma}\label[lemma]{c2}
    $\uap{\uap A \parr\to A}\parr= \top$.
\end{lemma}
\begin{proof}
    Let $E\in\lcalzpar$ be such that $\vdashc (\uap A \parr\to A)\to E$. 
    It suffices to show that $\vdashc E$.
    By assumption, we have $\vdashc A\to E$, and thus 
    $\vdashc \uap A \parr\to E$. On the other hand, $\vdashc (\uap A \parr\to A)\to E$ also implies 
    $\vdashc \neg \uap A \parr \to E$. Thus $\vdashc E$.
\end{proof}
\begin{theorem}\label[theorem]{classical-proj-all-2}
   All formulas are $\parr$-projective. Moreover, $\uap A\parr\to A$ for every $A$
   is projective.
\end{theorem}
\begin{proof}
Given $A\in\lcalz$, let $B:=\uap A \lcalzpar \to A$. 
\Cref{c2} implies that $\uap B\parr=\top$. Hence \Cref{c1} implies that $B$ is 
projective. Thus, \Cref{lem-proj-par-proj} gives the desired result.
\end{proof}
Recall from \cref{sec-definition} that a $\parr$-extension of a logic is an extension by an
$\lcalzpar$-formula.
\begin{corollary}\label[corollary]{unitary-cl-ext}
The unification type of $\parr$-extensions of classical logic is unitary.
\end{corollary}
\begin{proof}
Let $A\in\lcalz$ and $E\in\lcalzpar$ and assume that 
$B:=E\to A$ is unifiable, seeking to find a single unifier of 
$B$ that is more general than all of its unifiers.
Since $B$ is unifiable, there is some $\gamma$ such that $\vdashc \gamma(B)$,
and hence $\uap {B}\parr=\top$. Then 
\Cref{c1} implies that $B$ is projective. Let $\theta$ be its projective unifier. 
We claim that $\theta$ is also its most general unifier. 
Take any unifier $\tau$ of $B$. Since $B\vdashc \theta(x)\lr x$, we get 
$\tau(B)\vdashc \tau\theta(x)\lr \tau(x)$ and thus $\vdashc \tau\theta(x)\lr \tau(x)$.
This shows that $\tau$ is less general than $\theta$.
\end{proof}

\section{\texorpdfstring{$\Gamma$}{G}-Projectivity for Intuitionistic Logic}
As we have shown in previous section, classical logic is  well-behaving 
for $\parr$-projectivity in the following sense: all formulas are classically 
$\parr$-projective (\Cref{classical-proj-all-2}). 
Nevertheless, such a nice behaviour does not hold in
intuitionistic logic. For example the formula $A:=x\vee\neg x$, is not 
$\parr$-projective by the following argument. 
First, observe that $\uap A \lcalzpar =\top$. 
Hence, a parametrifier of $A$ should be 
a unifier. By the disjunction property for intuitionistic logic, $A$ has only two 
unifiers: those who replace the variable $x$ by truth, and by falsity respectively; namely 
$\theta_1(x):=\top$ and $\theta_2(x):=\bot$. This means that $A$  does not have a 
most general unifier. On the other hand, if $A$
were $\top$-projective, its projective unifier should be a most general unifier, a contradiction with our previous observation. Thus, $A$ is not $\parr$-projective.

Given that not all formulas are $\parr$-projective, the question arises on which formulas 
are $\parr$-projective.
In this section, we will characterise $\parr$-projectivity, via Kripke semantics. 
It is a variant of 
\cite{Ghil99} in which it is proven that projectivity is equivalent to \textit{extendability}
for a given formula in the non-parametric language.\footnote{Roughly speaking, an extendable formula is a formula whose Kripke models 
could be extended from below.}  
Let us first look at some definitions.

Given $X\subseteq\atom$, we say that $\kcal_1$ is a $X$-variant of $\kcal_2$, if 
    (1) they share the same frame, 
    (2) the evaluations at every node other than the root are the same, and 
    (3) the evaluations at the root for all $a\in X$ are the same. 
    An $\emptyset$-variant is simply called a variant.
    We say that $A$ is weakly valid in $\kcal$, notation $\kcal\vdashm A$, 
    if $A$ is valid on every node other than the root.

\begin{definition}\label[definition]{def-Gamma-extendable-formula}
We say that $A$ is $B$-extendable, if 
$\vdashi A\to B$ and
for every Kripke model $\kcal$ with $\kcal\vdashm A$ and $\kcal\Vdash B$, 
there is some $\parr$-variant $\kcal'$ of $\kcal$ such that $\kcal'\Vdash A$.
Additionally, for a given set $\Gamma$ of formulas, we say that $A$ is $\Gamma$-extendable, if 
there is some $B\in\Gpar$ such that $A$ is $B$-extendable.
\end{definition}

Given $A\in\lcalz$ and a set of variables $X$, we define $\theta^X _A$ as follows:
$$
\theta^X _A(x):=
\begin{cases}
A \to x  &: x\in X\\
A\wedge x &: x \nin X 
\end{cases}
$$

\begin{theorem}\label[theorem]{par-proj-par-extendable}
$E$-extendability and $E$-projectivity are equivalent notions 
for every $E\in\lcalzpar$.
\end{theorem}
\begin{proof}
    Right-to-left: Let $\theta$ be a projective $E$-fier of $A$, 
    i.e.~$A\vdash \theta(x)\lr x$ for every $x\in\varr$ and $\vdash\theta(A)\lr  E$.
    We will show that $A$ is $E$-extendable. Let $\kcal\vdashm A$ with 
    $\kcal\Vdash E$ and take $\kcal':=\theta(\kcal)$. 
    Then it is not difficult to observe that $\kcal'$ is indeed a $\parr$-variant of $\kcal$ and $\kcal'\Vdash A$.
    \\[2mm]
    Left-to-Right:  Assume that $A$ is $E$-extendable. Hence $\vdashi A\to E$ and for every $\kcal$ with 
    $\kcal\Vdash E$ and $\kcal\vdashm A$, there is a $\parr$-variant $\kcal'$ of $\kcal$ such that 
    $\kcal'\Vdash A$. 
    We will construct a substitution $\theta$, which is a projective $E$-fier of $A$.
    Our construction is by iterated compositions of substitutions $\theta^X _A$, and is the same as the one first introduced by S.~Ghilardi in \cite{Ghil99}.
    
    Let $X:= \{ x_1 , x_2, \ldots, x_m\} $ be the set of all variables occurring in $A$.
    Consider an ordering of the 
    powerset of $X$, namely $P(X)= \{ X_0, X_1,\ldots, X_k\} $ such that 
    $X_i \subseteq X_j \Rightarrow i \geq j$. Then 
    define $\theta_A := \theta_A^{X_0} \circ \theta_A^{X_1} \circ \ldots \circ \theta_A^{X_k}$ 
    which is an $A$-projection, since it is composition of $A$-projections. 
    We claim that $\theta_A$ is a projective $E$-fier of $A$. 
    Since $\theta_A$ is a cumulative composition of $A$-identity substitutions, and $A$-identity substitutions are closed under
    compositions, we may infer that $\theta_A$ is also $A$-identity. So, it only remains to show that
    $\theta_A$ is $E$-fier of $A$.
    By induction on the height $n$ of the Kripke model $\kcal$, we prove $\kcal\Vdash \theta_A(A)\lr E$.

    As induction hypothesis, let $\kcal\vdashm \theta_A(A)\lr E$.
    Given a node $w$ of $\kcal$, let $\gen\kcal w$ be the restriction of $\kcal$ to all nodes that are 
    accessible from $w$, including $w$ itself. Then, by the induction hypothesis, 
    we have $\gen\kcal w \Vdash \theta_A(A)\lr E $ for every $w$ other than the root.
    If $\kcal\nVdash E $, since $\vdashi \theta_A(A)\to E $, we get $\kcal\Vdash \theta_A(A)\lr E $.
    So, we may assume that $\kcal\Vdash E $. 
    Therefore, by $E$-extendability of $A$, there is a  $\parr$-variant $\kcal'$ 
    of $\kcal$ such that $\kcal'\Vdash A$. 
    Then \Cref{lem-ghil} implies that $\kcal\Vdash \theta_A(A)$ and so $\kcal\Vdash \theta_A(A)\lr E $.
\end{proof}

\begin{corollary}\label[corollary]{lem-par-proj-proj}
$\parr$-projectivity of $A$ is equivalent to projectivity of $\uap A \parr\to A$.
\end{corollary}
\begin{proof}
One direction has already been proved in \Cref{lem-proj-par-proj}. For the other way around,
assume that $A$ is $\parr$-projective and, hence, $\uap A \parr$-projective.
Then by \Cref{par-proj-par-extendable}, $A$ is also $\uap A \parr$-extendable. Again, by \Cref{par-proj-par-extendable},
it is enough to show that $\uap A\parr\to A$ is $\top$-extendable.  
So let $\kcal\vdashm \uap A\parr\to A$,
seeking to find some $\parr$-variant $\kcal'$ of $\kcal$ such that $\kcal'\Vdash \uap A\parr\to A$.
If $\kcal\nVdash \uap A\parr$, take $\kcal':=\kcal$ and we are done. Otherwise, by $\uap A \parr$-extendability 
of $A$, there is some $\parr$-variant $\kcal'$ of $\kcal$ such that $\kcal'\Vdash A$, as desired.
\end{proof}
\noindent
{\bf Question 1:} 
With the above corollary, we have a nice characterisation of projectivity of 
$\uap A \parr \to A$. We may wonder about (characterisation of) its dual notion, namely projectivity of $A\to \dap A\parr$, in which $\dap A\parr$ indicates the uniform pre-interpolant of $A$. 
\begin{lemma}\label[lemma]{lem-ghil}
    Let $\kcal'\Vdash \theta_A(A)$ is a $\parr$-variant of $\kcal$ and $\kcal\vdashm\theta_A(A)$. Then $\kcal\Vdash\theta_A(A)$. 
\end{lemma}
\begin{proof}
    We first make the observation that for any model $\kcal\not \Vdash A$ then for any set $Y$ and any $x\in Y\cap \varr$, $\theta ^{Y}_A(\kcal)\Vdash x$ iff $x\in Y$. Now for every $i\leq k$ let $\theta_A^i =\theta ^{X_0}_A \circ \ldots \circ \theta ^{X_i}_A $, 
    which are all $A$-identities. 
    As such, if $\kcal \Vdash \theta_A^i(A)$ for some $i\leq k$, then 
    $\kcal \Vdash \theta_A^j(A)$ for all $i\leq  j \leq k$. 
    Hence, it is sufficient to show that if $\kcal\vdashm \theta_A^i(A)$  
    and $\kcal'\Vdash \theta_A^j(A)$   with $i\leq j \leq k$, then $\kcal \Vdash \theta_A^r(A)$ for some $r\leq k$. 
    We prove this by induction on the height $n$ of $\kcal$.
    \begin{itemize}[leftmargin=*]
        \item $n=0$. Let $Y$ be the valuation of $\theta_A^i(\kcal')$ at its root and $Y\cap\varr = X_j$ for some $j$. 
        If $\theta^{j-1}_A (\kcal) \Vdash A$, then this step is done. If not, then for every variable $x \in X$, 
        $\theta^{j}_A (\kcal) \Vdash x $ iff $x \in Y$, and so $\theta^{j}_A (\kcal) \Vdash A$.
        \item Assume the induction claim for every model of height $k<n$ and let $\kcal$ be a model of height $n$ such that 
        $\kcal\vdashm \theta^i_A(A)$ with $i$ being least with this property. 
        Let $j$ be least 
        for which there exists a $\parr$-variant $\kcal'$ of $\kcal$ with 
        $\kcal'\Vdash \theta_A^j(A)$, then $ j\geq i$ as 
        $\kcal'\Vdash \theta_A^j(A)$ implies that $\kcal'\vdashm \theta_A^j(A)$ 
        hence $\kcal\vdashm \theta_A^j(A)$ and so $i\leq j$.
        
        Let $Y$ be the valuation of $\theta_A^j(\kcal')$ at its root and let $r$ be such that $Y\cap\,\varr = X_r$.  
        If $r< j$, then since $\theta_A^{j-1}(\kcal')\not\Vdash A$, we get for every variable $x\in X$ that 
        $\theta^j_A(\kcal')\Vdash x$ iff $x \in X_j$, a contradiction. Hence $Y\cap\varr = X_r $ for some $r\geq j$.
        If $j=r$ or $\theta ^{r-1}_A (\kcal) \Vdash A$ then $\theta ^{r}_A (\kcal) \Vdash A$. 
        Otherwise, since $\theta ^{r-1}_A (\kcal) \not\Vdash A$ we once more have for every variable $x \in X$ that 
        $\theta ^{r}_A (\kcal) \Vdash x $ iff $x \in Y\cap \varr $ and so $\theta ^{r}_A (\kcal) \Vdash A$.
    \end{itemize}
\end{proof}

\section{Unification type of extensions of intuitionistic logic}\label{sec-rel-proj-approximation}
Thanks to \cite{Ghil99}, we know that the unification type of 
intuitionistic logic is finitary. 
In this section, we show in \Cref{par-proj-approx} that 
the unification type of $\parr$-extensions of intuitionistic logic is also finitary.
\\[1mm]
\textbf{Convention.} All over this section, we assume that  $\atom$, the set of atomic formulas  is finite. 
\\[1mm]
We say that two Kripke models are $\parr$-equivalent ($\kcal_1\equiv^\parr\kcal_2$), 
if they share the same frame and for every node $w$ and 
every $p\in\parr$ we have $\kcal_1,w\Vdash p$ iff $\kcal_2,w\Vdash p$.

Given a class $\mathscr M$ of Kripke models, 
$\sum \mathscr M$ indicates the disjoint union of Kripke models in $\mathscr M$ with a fresh root
below with empty valuation. 
\begin{definition}\label[definition]{def-Gamma-extendable}
A class of Kripke models $\mathscr K$ is called $B$-extendable, if: 
$\mathscr K\Vdash B$ and 
     for every finite (including empty) 
    $\mathscr M\subseteq\mathscr K$, if  a variant $\kcal$ of $\sum \mathscr M$
    forces $B$, then there is a  $\parr$-variant of  $\kcal$ which belongs to $\mathscr K$. 
We say that $\mathscr K$ is $\Gamma$-extendable, if there is some $B\in\Gpar$ such 
that $\mathscr K$ is $B$-extendable.
    Then $\top$-extendability is also called extendability. 
It is not difficult to observe that $A$ is $B$-extendable 
(\Cref{def-Gamma-extendable-formula}) iff its class of models is so.
\end{definition}

Let $\complexity A$ indicate the maximum number of nested implications in $A$. More precisely,
\begin{itemize}
    \item $\complexity A=0$ for $A\in\atom\cup\{\bot\}$.
    \item $\complexity {A\wedge B}=\complexity{A\vee B}=\max\{\complexity A,\complexity B\}$.
    \item $\complexity{A\to B}:=1+\max\{\complexity A,\complexity B\}$.
\end{itemize}
In what follows, we will make use of the notions of bounded bisimilarity and bounded sub-bisimilarity. 
We define these notions inductively for given models $\kcal$ and $\kcal'$ with respective roots $w_0$ and $w_0'$:
\begin{itemize}
    \item $\kcal \bisim{0} \kcal'$ iff $\Val{\kcal}{w_0} = \Val{\kcal'}{w_0'}$.
    \item $\kcal \bisim{n+1} \kcal'$ iff 
    \begin{itemize}
        \item {\sf Forth:} $ \forall\, w \in \kcal\ \exists\, w' \in \kcal'\ \gen{\kcal}{w} \bisim{n} \gen{\kcal'}{w'} $ and
        \item {\sf Back:} $ \forall\, w' \in \kcal'\ \exists\, w \in \kcal\ \gen{\kcal}{w} \bisim{n} \gen{\kcal'}{w'} $.
    \end{itemize}
    \item $\kcal \leq_0 \kcal'$ iff $\Val{\kcal}{w_0} \supseteq \Val{\kcal'}{w_0'}$.
    \item $\kcal \leq_{n+1} \kcal'$ iff
    \begin{itemize}
        \item {\sf Forth:} $ \forall\, w \in \kcal\ \exists\, w' \in \kcal'\ \gen{\kcal}{w} \bisim{n} \gen{\kcal'}{w'} $.
    \end{itemize}
\end{itemize}
\begin{lemma}\label[lemma]{36}
    Let $\kcal\leq_n\kcal'$ and $\complexity A\leq n$.
    If $\kcal'\Vdash A$, then $\kcal\Vdash A$.
\end{lemma}
\begin{proof}
    See \cite{Ghil99}.
\end{proof}

Then for a class of Kripke models $\mathscr K$, 
we define 
$$\dclosur{\mathscr K} n:=\{\kcal: \exists \kcal'\in\mathscr K\ (\kcal\leq_n\kcal')\}.$$
We call a class $\scrk$ of Kripke models \emph{stable}, iff for every
$\kcal \in \scrk$ and every node $w$ of $\kcal$, $\gen{\kcal}{w}\in\scrk$. 
Furthermore, we say that $\scrk$ is $\leq_n$-downward closed, if 
$\kcal'\leq_n\kcal\in  \scrk  $ implies 
$\kcal'\in \scrk $.
Observe that if $\dclosur{\scrk} n = \scrk$ then $\scrk$ is stable and 
$\leq_n$-downward closed.

\begin{lemma}\label[lemma]{extendable-downward-closur}
Let $\scrk$ be a stable class of $E$-extendable 
Kripke models for some $E\in\lcalzpar$. 
Then for every $n> \complexity{E}$, $\dclosur{\mathscr K} n$ is also
stable and $E$-extendable.
\end{lemma}
\begin{proof}
By our above observation, we only need to show that $\dclosur{\scrk} n$ is $E$-extendable. 
Since $\scrk \Vdash E$ and $\complexity E < n$, we get from \Cref{36} that $\dclosur{\scrk} n \Vdash E $. 
Let $\scrm$ be a finite subset of $\dclosur{\mathscr K} n$;
then due to the stability of $\scrk$, there is a minimal set
$\scrm'\subseteq \scrk$ such that for every $ \mcal \in \scrm$ 
and every $w \in \mcal$ there is a model $\mcal' \in \scrm'$ 
such that $\gen \mcal w \bisim{n-1} \mcal'$. 
W.l.o.g. $\scrm'$ will be finite as there can only be at most finitely many non-$\bisim{n-1}$ Kripke models.
Assume that there is a variant $\kcal$ of $\sum \scrm $ such that $\kcal \Vdash E$ and let $\kcal'$ be a variant of $\sum \scrm '$ with its root evaluated the same as $\kcal$.
We show that $\kcal \bisim{n-1} \kcal'$.
\begin{itemize}
    \item By our assumption for $\mcal'$ for every $w$ other than the root, there is some $\mcal' \in \scrm'$ such that $\gen\kcal w \bisim{n-2} \mcal'$ and vice versa.
    \item The roots are $0$-bisimilar to each other and, by the above, $\kcal \bisim{n-1} \kcal'$.
\end{itemize}
Hence $\kcal' \Vdash E$, so there is a $\parr$-variant $\overline{\kcal'} \in \scrk$ of $\kcal'$.
Consider the $\parr$-variant $\overline{\kcal}$ of $\kcal$ whose root is evaluated the same as in $\overline{\kcal'}$.
By omitting the backward direction in the proof of $\kcal \bisim{n-1} \kcal'$, one can easily show that $\overline{\kcal} \leq_n \overline{\kcal'}$.
\end{proof}

\begin{corollary}\label[corollary]{extendable-downward-closure-lc}
Let $\Gamma$ be finite. Then there is $m$ such that for every stable class $\scrk$ of 
$\Gamma$-extendable Kripke models and every $n> m$, $\dclosur{\mathscr K} n$ is also
stable and $\Gamma$-extendable.
\end{corollary}
\begin{proof}
    Let $m=\max\{ \complexity{E}: E\in\Gpar\}$ and apply \Cref{extendable-downward-closur}.
\end{proof}

\begin{lemma}\label[lemma]{modA-downward-closed}
    A class $\mathscr K$ of Kripke models is equal to the class of all models of some $A$ with $\complexity A \leq n$ iff 
    $\mathscr K$ is $\leq_n$-downward closed.
\end{lemma}
\begin{proof}
    See \cite{Ghil99}.
\end{proof}

\begin{theorem}\label[theorem]{par-proj-approx-general}
Given a finite $\Gamma$ and $A$, there is a finite set $\Pi$ of $\Gamma$-projective formulas with
the following properties:
\begin{enumerate}
    \item $\vdashi\bigvee\Pi\to A$.
    \item $A\adms{}\Gamma \Pi$. (see \cref{sec-admissibility} for the definition of $\adms{}\Gamma$.)
\end{enumerate}
\end{theorem}
\begin{proof}
Let  $m_0:=\complexity A$ and $m_1$ be the number $m$ given from 
\Cref{extendable-downward-closure-lc} and $n:=\max\{ m_0 , m_1\}$. 
For a given $E\in\Gpar$ and substitution $\theta$ with  
$\vdashi E\to\theta(A)$, we find some $B_E^\theta\in\lcalz$ such that:
\begin{enumerate}
    \item\label{par-proj-approx-1} $\complexity{B_E^\theta}\leq n$.
    \item\label{par-proj-approx-2} $B_E^\theta$ is $E$-projective.
    \item\label{par-proj-approx-3} $\vdashi B_E^\theta\to A$.
    \item\label{par-proj-approx-4} $\vdashi E\to \theta(B_E^\theta)$.
\end{enumerate}
Let us first see why this finishes the proof of this theorem. 
Define 
$$
\Pi:=\{B_E^\theta: E\in\Gpar \ \& \ \vdashi E\to\theta(A)\}.
$$
Since $\atom$ is finite, by
 (\ref{par-proj-approx-1}) we have the finiteness of $\Pi$. 
$\Gamma$-projectivity of elements of $\Pi$ is guaranteed by 
(\ref{par-proj-approx-2}). Item (\ref{par-proj-approx-3}) implies
$\vdashi \bigvee\Pi\to A$. Finally, by item (\ref{par-proj-approx-4}) 
we have $A\admspar\Pi$.

So, let us go back to find $B_E^\theta$ with the listed properties.
Define 
$$\mathscr K:=\{\theta(\kcal): \kcal\Vdash E\}.$$
It is easy to observe that $\mathscr K$ is stable and $E$-extendable.
Then \Cref{extendable-downward-closur} implies that
$\dclosur{\mathscr K} n$ is also stable and $E$-extendable. 
\Cref{modA-downward-closed} implies that there is some 
formula $B_E^\theta$ with $\complexity{B^\theta_E}\leq n$ such that 
$\dclosur{\mathscr K}n=\{\kcal: \kcal\Vdash B^\theta_E\}$. 
Hence (\ref{par-proj-approx-1}) is satisfied.
Since $\dclosur{\mathscr K} n$ is extendable, 
\Cref{par-proj-par-extendable} implies the validity of item (\ref{par-proj-approx-2}).
To check the validity of item (\ref{par-proj-approx-3}), we argue semantically. Let
$\kcal\Vdash B^\theta_E$. Therefore, $\kcal\leq_n\kcal'$ for some
$\kcal'\in\mathscr K$. Since $\kcal'\in\mathscr K$, 
we have $\kcal'=\theta(\kcal'')$ for some $\kcal''\Vdash E$. 
Since $\kcal''\Vdash E$ and $E\in\lcalzpar$, we also
have $\theta(\kcal'')\Vdash E$. On the other hand, since
$\vdashi E\to\theta(A)$, we have $\kcal''\Vdash E\to \theta(A)$.
Hence $\kcal''\Vdash \theta(A)$ and thus $\theta(\kcal'')\Vdash A$.
This means $\kcal'\Vdash A$ and since $\complexity A\leq n$  
and $\kcal\leq_n\kcal'$, \Cref{36} implies $\kcal\Vdash A$. Thus,
$\vdashi B^\theta_E\to A$.
For the item (\ref{par-proj-approx-4}), let $\kcal\Vdash E$.
Then $\theta(\kcal)\in\mathscr K\subseteq\dclosur{\mathscr K}n$.
Hence $\theta(\kcal)\Vdash B^\theta_E$ and thus $\kcal\Vdash \theta(B^\theta_E)$.
\end{proof}

Recall from \cref{sec-definition} that a $\parr$-extension of a logic is an extension by an
$\lcalzpar$-formula. 
\begin{theorem}\label[theorem]{par-proj-approx}
The unification type of $\parr$-extensions of intuitionistic logic is finitary. 
\end{theorem}
\begin{proof}
Let $A\in\lcalz$ and $E\in\lcalzpar$ be given, seeking a finite complete set of unifiers 
of $A$ in the intuitionistic logic extended by $E$.
Take $\Gamma:=\{E\}$ and apply 
\Cref{par-proj-approx-general} to obtain a finite set $\Pi$ with the properties mentioned. Then it can
be easily verified that the set of 
projective $E$-fiers of elements of $\Pi$ serves as 
a finite complete set of unifiers, as desired.
\end{proof}

\section{Parametric Admissibility: Relative to  \texorpdfstring{$\Gamma$}{G}}\label[section]{sec-admissibility}

Given a logic $\sfL$, the (multi-conclusion) 
admissibility relation for $\sfL$ is defined as follows:
$$A\adms{}\sfL \Delta \quad \text{iff}\quad 
\forall\theta\ (\sfL\vdash\theta(A)\ \Rightarrow\ \exists\, B\in\Delta\ \   \sfL\vdash\theta(B)). $$
We simply write $A\adms{}\sfL B$ for $A\adms{}\sfL \{B\}$.
Given $E\in\lcalzpar$, we define $\adms{}E$ as the admissibility relation for intuitionistic logic strengthened by an additional axiom $E$. Note that since 
$E\in\lcalzpar$, this logic is closed under substitutions, and hence it is a logic indeed.
The admissibility for intuitionistic logic $\adms{}\IPC$ is annotated as $\adms{}{}$ in 
the literature \cite{Iemhoff-admissibility} and is known to be decidable \cite{Rybakov87}.
One may easily observe that
$A\adms{}{E} B$ iff $(E\to A)\adms{}{}(E\to B)$ for every $E\in\lcalzpar$. 
Thus, by characterising the admissibility relation $\adms{}{}$, we also have 
a characterisation of $\adms{}{E}$. 

Then we define the $\Gamma$-admissibility relation $\adms{}\Gamma$, as the intersection of all
$\adms{}E$'s with $E\in\Gpar=\Gamma\cap\parr$.
In other words:
$$
A\adms{}\Gamma B
\quad \text{iff}\quad 
\forall\,E\in\Gpar\  A\adms{}E B.
$$
This time, the characterisation of $\adms{}\Gamma$ is not that simple. 
For instance, the characterisation for $\Gamma=\NNIL$, the set of No Nested Implications on the Left,
is the main result of \cite{PLHA0}. This result has been used in an essential way in the proof for 
characterisation of intuitionistic provability logic \cite{PLHA}. In this paper, we will dig further and characterise one more salient case: $\Gamma=\lcalz$, the full language. 
It turns out that $A\adms{}{\lcalz}B$ iff $\vdashi A\to B$ (\Cref{cor-charac-parr-admissibility}).

There is yet another binary relation on $\lcalz$, called preservativity \cite{Visser02} and 
annotated as $A\pr_\Gamma B$, which is defined as follows:
$$A\pr_\Gamma B \quad \text{iff}\quad \forall\, E\in\Gamma\ (\vdashi E\to A \ \Longrightarrow \ \vdashi E\to B).$$
Note the difference between the role of $E$ in $\adms{}{\Gamma} $ and $\pr_\Gamma$. In the first, we quantify $E$ 
over the set $\Gpar$, a subset of $\Gamma$, whereas in the second, we quantify it over $\Gamma$ itself. Since both definitions are universal over $\Gamma$, the following trivially holds:
\begin{remark}
    If $\Gamma \subseteq \Delta$ then 
    ${\adms{}{\Gamma}}\supseteq {\adms{}{\Delta}}$ and 
    ${\pr_{\Gamma}}\supseteq {\pr_{\Delta}}$.
\end{remark}
In the following lemma, we show that $\Gamma$-admissibility 
implies $\proj\Gamma$-preservativity, where $\proj\Gamma$ is the set of $\Gamma$-projective formulas. 
\begin{lemma}\label[lemma]{lem-adms-impl-pres}
    ${\adms{}{\Gamma}}\subseteq{\pr_{\proj\Gamma}}$.
\end{lemma}
\begin{proof}
    Let $E\in \proj\Gamma$ such that $\vdashi E\to A$, seeking 
to show $\vdashi E\to B$. 
Since $E\in\proj\Gamma$, there is a substitution $\theta$
that projects $E$ to $\gampro E \Gamma\in\lcalzpar$.
Therefore, by $\vdashi E\to A$, we have $\vdashi \gampro E\Gamma\to \theta(A)$. Hence, by $A\adms{}\Gamma B $ we have
$\vdashi \gampro E\Gamma\to \theta(B)$. Since $\vdashi E\to \gampro E\Gamma$, we have
$\vdashi E\to  \theta(B)$, and thus by the $E$-projectivity of $\theta$, we have $\vdashi E\to B$.
\end{proof}

A $\parr$-substitution is a substitution $\theta$ such that $\theta(x)\in\parr$
for every $x\in\varr$. Then we say that $\Gamma$ is closed under 
$\parr$-substitutions if for every $A\in\Gamma$ and every $\parr$-substitution
$\theta$ we have $\theta(A)\in\Gamma$.
\begin{lemma}\label[lemma]{lem-charac-parr-admissibility}
    Given $\Gamma$ closed under $\parr$-substitutions and $A\in\Gamma$, 
    the following are equivalent:
\begin{enumerate}
    \item $\vdashi A\to B$.
    \item $A\adms{}{\Gamma} B$.
    \item $A\pr_{\proj\Gamma} B$.
\end{enumerate}
\end{lemma}
\begin{proof}
The proof of $1\to 2$ is obvious. 
$2\to 3$ holds by \Cref{lem-adms-impl-pres}.
We reason for $3\to 1$ as follows. 
    Let $\nvdashi A\to B$. Also, assume that $X:=\{x_1,\ldots,x_n\}$ includes all variables 
    occurring in $A$, and $\{p_1,\ldots,p_n\}$ is a set of fresh parameters, i.e.~parameters not appeared in $A$ and $B$. Finally, assume that 
    $$E:=A\wedge \bigwedge_{i=1}^n x_i\lr p_i\quad \text{and}\quad \theta(x):=\begin{cases}
    p_i \quad &: x\in X \text{ and } x=x_i \text{ for } 1\leq i \leq n\\
    x &: x\nin X
    \end{cases} $$
It can be easily verified that 
\begin{itemize}
    \item $E$ is $\Gamma$-projective.
\item $\vdashi E\to A$.
\item $\nvdashi E\to B$.
\end{itemize}
This implies that $A\not\pr_{\proj\Gamma} B$, as desired.
\end{proof}
Note that in the proof of the above lemma,
we are taking advantage of the infiniteness of $\parr$.  

As a result, if we take $\Gamma$ as the full language 
$\lcalz$, we have the following:
\begin{corollary}\label[corollary]{cor-charac-parr-admissibility}
The following are equivalent:
\begin{enumerate}
    \item $\vdashi A\to B$.
    \item $A\adms{}{\lcalz} B$.
    \item $A\pr_{\proj\lcalz} B$.
\end{enumerate}    
\end{corollary}

Given that $A\adms{}{\lcalz}B$ is equivalent to $\vdashi A\to B$, asking for $\parr$-projective approximations of a proposition $A$ (\Cref{par-proj-approx-general}), simplifies 
to the following question.\\[2mm]
{\bf Question 2:} Given $A\in\lcalz$, is it possible to find a finite set $\Pi_A$ of $\lcalz$-projective propositions with $$\vdashi A\lr\bigvee\Pi_A.$$
\\[2mm]
By the arguments in this section, we already have 
a characterisation / decidability of $\adms{}{\Gamma}$ 
for finite $\Gamma$ and for $\Gamma=\lcalz$. 
Moreover, \cite{PLHA0} provides a characterisation for the case $\Gamma=\NNIL$, the set of No Nested Implications on the Left.\footnote{Since its appearance in the literature 
\cite{Visser-Benthem-NNIL}, 
$\NNIL$ has shown itself to 
be of great importance in the study of
intuitionistic logic.} 
\\[2mm]
{\bf Question 3:}  
Axiomatize or provide decision algorithm for 
$\adms{}\Gamma$ in the following cases for $\Gamma$:
\begin{enumerate}
    \item \textit{Weakly extendable} formulae: We say that $A$ is weakly extendable if every $\kcal$ with $\kcal\vdashm A$ has a variant $\kcal'$ such that $\kcal'\Vdash A$.
    \item \textit{Prime}
formulae:  A formula $A$ is prime if $\vdashi A\to(B\vee C)$ implies $\vdashi A\to B$ or $\vdashi A\to C$. Note that every extendable formula is prime, but not necessarily vice versa.\footnote{For example
the formula $\neg p \to (q\vee r)$ is prime, while it is not extendable.}
\end{enumerate}


\bibliographystyle{aiml}
\bibliography{aiml}

\end{document}